\documentclass[reqno,11pt]{amsart}

\usepackage{amssymb}
\usepackage{amsgen}
\usepackage{amsmath}
\usepackage{amsthm}
\usepackage{mathrsfs}
\usepackage{cite}
\usepackage{amsfonts}

\newcommand{\Sym}{\mathop{\mathcal S}\nolimits}

\newcommand{\inv}{^{-1}}

\usepackage{xcolor}

\newtheorem{Thm}{Theorem}

\newtheorem{Lemma}[Thm]{Lemma}
{\theoremstyle{definition}
}
{\theoremstyle{remark}
\newtheorem{Rmk}{Remark}}

{\theoremstyle{remark}
}
{\theoremstyle{remark}
}

{\theoremstyle{remark}
}
{\theoremstyle{remark}
}
{\theoremstyle{remark}
}

\title{On the Burnside-Brauer-Steinberg theorem}

\author{Benjamin Steinberg}
\address{%
    Department of Mathematics\\
    City College of New York\\
    Convent Avenue at 138th Street\\
    New York, New York 10031\\
    USA}
\email{bsteinberg@ccny.cuny.edu}

\thanks{This work was partially supported by a grant from the Simons Foundation(\#245268
to Benjamin Steinberg), the Binational Science Foundation of Israel and the US (\#2012080 to Benjamin Steinberg) and by a CUNY Collaborative Incentive Research Grant.}
\date{\today}

\keywords{monoids, representation theory, characters, tensor products, symmetric powers}
\subjclass[2010]{20M30,20C15,16G99,16T10}

\begin{document}

\begin{abstract}
A well-known theorem of Burnside says that if $\rho$ is a faithful representation of a finite group $G$ over a field of characteristic $0$, then every irreducible representation of $G$ appears as a constituent of a tensor power of $\rho$.  In 1962, R.~Steinberg gave a module theoretic proof that simultaneously removed the constraint on the characteristic, and allowed the group to be replaced by a monoid.  Brauer subsequently simplified Burnside's proof and, moreover, showed that if the character of $\rho$ takes on $r$ distinct values, then the first $r$ tensor powers of $\rho$ already contain amongst them all of the irreducible representations of $G$ as constituents.  In this note we prove the analogue of Brauer's result for finite monoids.  We also prove the corresponding result for the symmetric powers of a faithful representation.
\end{abstract}

\maketitle

\section{Introduction}
A famous result of Burnside~\cite{Burnsidebook} states that if $K$ is a field of characteristic $0$, $G$ is a finite group and $V$ is a finite dimensional $KG$-module affording a faithful representation of $G$, then each simple $KG$-module is a composition factor of a tensor power $V^{\otimes i}$ of $V$. Burnside's original proof~\cite{Burnsidebook} was via characters and formal power series.  This result was vastly generalized by R.~Steinberg in 1962~\cite{RSteinberg}.  He showed that if $K$ is any field, $M$ is any monoid (possibly infinite) and $V$ is a $KM$-module affording a faithful representation of $M$, then the tensor algebra $T(V)=\bigoplus_{i=0}^{\infty} V^{\otimes i}$ is a faithful $KM$-module (i.e., its annihilator in $KM$ is $0$).  This easily implies that if $M$ is finite and $V$ is finite dimensional, then every simple $KM$-module is a composition factor of some tensor power of $V$ (in fact one of the first $|M|$).  Rieffel extended this result even further to bialgebras~\cite{Rieffelburn}; see also~\cite{Passman,Passman2}.

In 1964, Brauer gave a simpler character-theoretic proof of Burnside's theorem and at the same time refined it~\cite{BrauerBurn}.  Namely, he showed that if $G$ is a finite group, $K$ is a field of characteristic $0$ and $V$ is a finite dimensional $KG$-module affording a faithful representation of $G$ whose character takes on $r$ distinct values, then every simple $KG$-module is a composition factor of one of the first $r$ tensor powers of $V$.  Because of this refinement, Burnside's result is often referred to as the Burnside-Brauer theorem.

It is natural to ask whether R.~Steinberg's theorem can be similarly refined: is it true that if $V$ is a finite dimensional $KM$-module affording a faithful representation of a finite monoid $M$ over a field $K$ of characteristic $0$ and that the character of $V$ takes on only $r$ distinct values, then every simple $KM$-module is a composition factor of one of $V^{\otimes 0},\ldots, V^{\otimes (r-1)}$?

This note answers the above question affirmatively. On the other hand,  we also show that the minimal $k$ such that  $\bigoplus_{i=0}^{k} V^{\otimes i}$ is a faithful $KM$-module cannot be bounded as a function of solely the number of distinct values assumed by the character of $V$, as is the case for finite groups.

Brauer's proof~\cite{BrauerBurn} relies on the orthogonality relations for group characters.  The irreducible characters of a finite monoid do not form an orthogonal set with respect to the natural inner product on mappings $M\to K$. So we have to adopt a slightly different tactic.  Instead of using the orthogonality relations, we apply the character of $V^{\otimes i}$ to carefully chosen primitive idempotents.  To make Brauer's argument work, we also need to apply at a key moment a small part of the structure theory of irreducible representations of finite monoids, cf.~\cite{RhodesZalc,LallePet,gmsrep} and~\cite[Chapter~5]{CP}.

A detailed study of the minimal degree a faithful representation of a finite monoid was undertaken by the author and Mazorchuk in~\cite{effective}.

It is also known that if $V$ is a finite dimensional $KG$-module affording a faithful representation of a finite group $G$ over a field of characteristic $0$, then every simple $KG$-module is a composition factor of a symmetric power $\Sym^n(V)$ of $V$, cf.~\cite{etingof}.  We prove the corresponding result for monoids and give a bound on how many symmetric powers are needed in terms of $\dim V$ and the number of distinct characteristic polynomials of the linear operators associated to elements of $M$ acting on $V$.  These kinds of results for representations of finite monoids over finite fields can be found in~\cite{genericrep,Kuhn2}.

\section{Tensor powers}
We follow mostly here the terminology of the book of Curtis and Reiner~\cite{curtis}, which will also serve as our primary reference on the representation theory of finite groups and finite dimensional algebras.

Let $K$ be a field, $A$ a finite dimensional $K$-algebra, $S$ a simple $A$-module and $V$ a finite dimensional $A$-module.   We denote by $(V:S)$ the multiplicity of $S$ as a composition factor of $V$. Recall that $S\cong Ae/Re$ where $R$ is the radical of $A$ and $e\in A$ is a primitive idempotent, cf.~\cite[Corollary~54.13]{curtis}.  (An idempotent $e$ is \emph{primitive} if whenever $e=e_1+e_2$ with $e_1,e_2$ orthogonal idempotents, then either $e_1=0$ or $e_2=0$.)
To prove the main result, we need two lemmas about finite dimensional algebras.  The first is the content of~\cite[Theorem~54.12]{curtis}.

\begin{Lemma}\label{multiplicity}
Let $K$ be a field and $A$ a finite dimensional $K$-algebra with radical $R$.  Let $S$ be a simple $A$-module, $e\in A$ a primitive idempotent with $S\cong Ae/Re$ and $V$ a finite dimensional $A$-module.  Then $(V:S)>0$ if and only if $eV\neq 0$.
\end{Lemma}

The second lemma on finite dimensional algebras concerns the connection between primitive idempotents for an algebra and its corners.  We recall that if $A$ is a finite dimensional algebra with radical $R$ and $e\in A$ is an idempotent, then $eRe$ is the radical of $eAe$~\cite[Theorem~54.6]{curtis}.

\begin{Lemma}\label{corner}
Let $A$ be a finite dimensional $K$-algebra with radical $R$ and let $e\in A$ be an idempotent.  Suppose that $S$ is a simple $A$-module such that $eS\neq 0$.  Then $eS$ is a simple $eAe$-module and, moreover, if $f\in eAe$ is a primitive idempotent with $eAef/eRef\cong eS$, then $f$ is a primitive idempotent of $A$ and $Af/Rf\cong S$.
\end{Lemma}
\begin{proof}
If $v\in eS$ is a nonzero vector, then $eAev=eAv=eS$ because $S$ is a simple $A$-module. Thus $eS$ is a simple $eAe$-module.  Let $f\in eAe$ be as above.  If $f=e_1+e_2$ with $e_1,e_2$ orthogonal idempotents in $A$, then $ee_ie=efe_ife=fe_if=e_i$ for $i=1,2$ and so $e_1,e_2\in eAe$.  Thus one of $e_1,e_2$ is $0$ by primitivity of $f$ in $eAe$ and hence $f$ is primitive in $A$.  Finally, since $(eS:eAef/eRef)=1$, we have by Lemma~\ref{multiplicity} that $0\neq feS=fS$ and so $(S:Af/Rf)>0$ by another application of Lemma~\ref{multiplicity}. Since $S$ is simple, we deduce that $S\cong Af/Rf$, as required.
\end{proof}

Next we need a lemma about idempotents of group algebras.

\begin{Lemma}\label{groupalgidemp}
Let $G$ be a finite group and $K$ a field of characteristic $0$.  Suppose that $e=\sum_{g\in G}c_gg$ in $KG$ is a nonzero idempotent.  Then $c_1\neq 0$.
\end{Lemma}
\begin{proof}
Because $e\neq 0$, we have $\dim eKG>0$.  Let $\theta$ be the character of the regular representation of $G$ over $K$, which we extend linearly to $KG$.  Then \[\dim eKG=\theta(e) = \sum_{g\in G}c_g\theta(g) =c_1\cdot |G|\] since
\[\theta(g) = \begin{cases} |G|, & \text{if}\ g=1\\ 0, & \text{else.}\end{cases}\] Therefore, $c_1=(\dim eKG)/|G|\neq 0$.
\end{proof}

Let $M$ be a finite monoid and $K$ a field.  If $V$ is a finite dimensional $KM$-module, then $\theta_V\colon M\to K$ will denote the character of $V$. Sometimes it will be convenient to extend $\theta_V$ linearly to $KM$.
Note that $V^{\otimes i}$ is a $KM$-module by defining \[m(v_1\otimes\cdots\otimes v_i)=mv_1\otimes\cdots\otimes mv_i\] for $m\in M$.  By convention $V^{\otimes 0}$ is the trivial $KM$-module.
One has, of course, that $\theta_{V\otimes W}=\theta_V\cdot \theta_W$ and that the character of the trivial module is identically $1$.  Therefore, $\theta_{V^{\otimes i}}=\theta_V^i$ for all $i\geq 0$.
The following is a monoid analogue of a well-known fact for groups.

\begin{Lemma}\label{charkernel}
Let $M$ be a finite monoid, $K$ a field of characteristic $0$ and $\rho\colon M\to M_n(K)$ a representation affording the character $\theta$.  Then $\rho(m)=I$ if and only if $\theta(m)=n$.
\end{Lemma}
\begin{proof}
If $\rho(m)=I$, then trivially $\theta(m)=n$.  Suppose that $\theta(m)=n$.  Because $M$ is finite, there exist $r,s>0$ such that $m^r=m^{r+s}$.  Then the minimal polynomial of $\rho(m)$ divides $x^r(x^s-1)$ and so each nonzero eigenvalue of $\rho(m)$ is a root of unity (in an algebraic closure of $K$).  Now the proof proceeds analogously to the case of finite groups, cf.~\cite[Corollary~30.11]{curtis}. That is, $\theta(m)$ is a sum of at most $n$ roots of unity and hence can only be equal to $n$ if all the eigenvalues of $\rho(m)$ are $1$.  But then $\rho(m)$ is both unipotent and of finite order, and hence $\rho(m)=I$ as $K$ is of characteristic $0$.
\end{proof}

We shall now need to apply a snippet of the structure theory for irreducible representations of finite monoids. Details can be found in~\cite[Chapter~5]{CP} or~\cite{RhodesZalc}; a simpler approach was given in~\cite{gmsrep}.
Let $M$ be a finite monoid and $e\in M$ an idempotent.  Denote by $G_e$ the group of units of the monoid $eMe$.  It is well known that $I_e=eMe\setminus G_e$ is an ideal of $eMe$, i.e., $(eMe)I_e(eMe)=I_e$; see, for instance,~\cite[Proposition~1.2]{TilsonXI} in Eilenberg~\cite{Eilenberg}.

\begin{Lemma}\label{restrict}
Let $M$ be a monoid and $K$ a field. Let $e\in M$ be an idempotent and let $V$ be a finite dimensional $KM$-module.  Then $(\theta_V)|_{eMe}=\theta_{eV}$.
\end{Lemma}
\begin{proof}
There is a vector space direct sum decomposition $V=eV\oplus (1-e)V$. As $eMe$ annihilates $(1-e)V$ and preserves $eV$, the result follows.
\end{proof}

Let $S$ be a simple $KM$-module with $K$ a field.  An idempotent $e\in M$ is called an \emph{apex} for $S$ if $eS\neq 0$ and $I_eS=0$.  By classical results of  Munn~\cite{Munn1} and Ponizovsky~\cite{Poni}, each simple $KM$-module has an apex; see~\cite[Theorem~5]{gmsrep} or~\cite[Theorem~5.33]{CP}.  The apex is unique up to  $\mathscr J$-equivalence of idempotents, although this fact is not relevant here.  We are now ready to prove our refinement of R.~Steinberg's theorem~\cite{RSteinberg}.

\begin{Thm}\label{main}
Let $M$ be a finite monoid and $K$ a field of characteristic $0$. Let $V$ be a finite dimensional $KM$-module affording a faithful representation of $M$.  Suppose that the character $\theta$ of $V$ takes on $r$ distinct values.   Then every simple $KM$-module is a composition factor of $V^{\otimes i}$ for some $0\leq i\leq r-1$.
\end{Thm}
\begin{proof}
Let $S$ be a simple $KM$-module and let $e\in M$ be an apex for $S$. Put $A=KM$ and let $R$ be the radical of $A$. Observe that $eAe=K[eMe]$.  As $eS\neq 0$, there is a primitive idempotent $f$ of $eAe$ such that $f$ is primitive in $A$ and $S\cong Af/Rf$ by Lemma~\ref{corner}.  Write \[f=\sum_{m\in eMe}c_mm.\]  By definition of an apex $I_eS=0$.  On the other hand, $fS\neq 0$ by Lemma~\ref{multiplicity}.  Thus $f\notin KI_e$.  Define a homomorphism $\varphi\colon eAe\to KG_e$ by \[\varphi(m) = \begin{cases} m, & \text{if}\ m\in G_e\\ 0, & \text{if}\ m\in I_e\end{cases}\] for $m\in eMe$ and note that $\ker\varphi=KI_e$. Therefore, \[\varphi(f)=\sum_{g\in G_e} c_gg\] is a nonzero idempotent of $KG_e$ and hence $c_e\neq 0$ by Lemma~\ref{groupalgidemp}.

Let $\theta_1,\ldots,\theta_r$ be the values taken on by $\theta$ and let \[M_j=\{m\in eMe\mid \theta(m)=\theta_j\}.\]  Without loss of generality assume that $\theta_1=\theta(e)$. Put \[b_j = \sum_{m\in M_j} c_m.\]
Suppose now that $(V^{\otimes i}:S)=0$ for all $0\leq i\leq r-1$.  We follow here the convention that $\theta_j^0=1$ even if $\theta_j=0$. Then by Lemma~\ref{multiplicity}, we have that \[0=\dim fV^{\otimes i} = \theta_{V^{\otimes i}}(f) = \sum_{m\in eMe}c_m\theta^i(m) = \sum_{j=1}^r\theta_j^i\sum_{m\in M_j}c_m=\sum_{j=1}^r\theta_j^ib_j\] for all $0\leq i\leq r-1$.  By nonsingularity of the Vandermonde matrix, we conclude that $b_j=0$ for all $1\leq j\leq r$.  By Lemma~\ref{restrict} we have that $M_1=\{m\in eMe\mid \theta_{eV}(m)=\dim eV\}$.  Because $V$ affords a faithful representation of $M$, it follows that $eV$ affords a faithful representation of $eMe$.  Lemma~\ref{charkernel} then implies that $M_1=\{e\}$.  Thus $0=b_1=c_e\neq 0$.  This contradiction concludes the proof.
\end{proof}

\begin{Rmk}
We need to include the trivial representation $V^{\otimes 0}$ because if $M$ is a monoid with a zero element $z$ and if $zV=0$, then $zV^{\otimes i}=0$ for all $i>0$ and so the trivial representation is not a composition factor of any positive tensor power of $V$.  The proof of Theorem~\ref{main} can be modified to show that if $S$ is not the trivial module, or if $M$ has no zero element, then $S$ appears as a composition factor of $V^{\otimes i}$ with $1\leq i\leq r$. The key point is that only the trivial representation can have the zero element of $M$ as an apex and so in either of these two cases, $\theta(e)\neq 0$.
\end{Rmk}

\begin{Rmk}\label{nobound}
If $G$ is a finite group, $K$ is a field of characteristic $0$ and $V$ is a finite dimensional $KG$-module affording a faithful representation of $G$ whose character takes on $r$ distinct values, then $\bigoplus_{i=0}^{r-1} V^{\otimes i}$ contains every simple $KG$-module as a composition factor by Brauer's theorem and hence is a faithful $KG$-module because $KG$ is semisimple.  We observe that the analogous result fails in a very strong sense for monoids.
Let $N_t=\{0,1,\ldots, t\}$ where $1$ is the identity and $xy=0$ for $x,y\in N_t\setminus \{1\}$.  Define a faithful two-dimensional representation $\rho\colon N_t\to M_2(\mathbb C)$ by
\[\rho(0) = \begin{bmatrix} 0 & 0\\ 0& 0\end{bmatrix},\ \rho(1) = \begin{bmatrix} 1 & 0\\ 0& 1\end{bmatrix},\ \rho(j) = \begin{bmatrix} 0 & j\\ 0& 0\end{bmatrix},\ \text{for}\ 2\leq j\leq t.\]
Let $V$ be the corresponding $\mathbb CN_t$-module.  The character of $\rho$ takes on $2$ values, $0$ and $1$. However, $V^{\otimes 0}\oplus V^{\otimes 1}$ is $3$-dimensional and so cannot be a faithful $\mathbb CN_t$-module for $t\geq 9$ by dimension considerations.  In fact, given any integer $k\geq 0$, we can choose $t$ sufficiently large so that $\bigoplus_{i=0}^k V^{\otimes i}$ is not a faithful $\mathbb CN_t$-module (again by dimension considerations).  Thus, the minimum $k$ such that $\bigoplus_{i=0}^k V^{\otimes i}$ is a faithful $\mathbb CN_t$-module cannot be bounded as a function of only the number of values assumed by the character $\theta_V$ (independently of the monoid in question).
\end{Rmk}

\begin{Rmk}
A monoid homomorphism $\varphi\colon M\to N$ is called an \emph{$\mathbf{LI}$-morphism} if $\varphi$ separates $e$ from $eMe\setminus \{e\}$ for all idempotents $e\in M$.  The proof of Theorem~\ref{main} only uses that the representation $\rho\colon M\to \mathrm{End}_K(V)$ afforded by $V$ is an $\mathbf{LI}$-morphism, and not that it is faithful. Hence one could obtain the conclusion of Theorem~\ref{main} under the  weaker hypothesis that the representation afforded by $V$ is an $\mathbf{LI}$-morphism.  However, if $\varphi\colon M'\to M''$ is a surjective $\mathbf{LI}$-morphism of finite monoids and $K$ is a field of characteristic $0$, then the induced algebra homomorphism $\varphi\colon KM'\to KM''$ has nilpotent kernel~\cite{AMSV} and hence each simple $KM'$-module is lifted from a simple $KM''$-module.  Thus applying Theorem~\ref{main} to $\rho(M)$ allows one to recover the result under the weaker hypothesis from the original result.
\end{Rmk}

\section{Symmetric powers}
Let $K$ be a field of characteristic $0$ and $V$ a vector space over $K$.  Then the symmetric group $S_d$ acts on the right of  $V^{\otimes d}$ by twisting, e.g., \[(v_1\otimes \cdots \otimes v_d)\sigma = v_{\sigma\inv(1)}\otimes \cdots \otimes v_{\sigma\inv(d)}.\]   The $d^{th}$-symmetric power is the coinvariant space \[\Sym^d(V)=V^{\otimes d}\otimes_{KS_d} K\] where $K$ is the trivial $KS_d$-module.   In characteristic $0$, one can identify $\Sym^d(V)$ with the symmetric tensors (the tensors fixed by $S_d$).  If $V$ is a $KM$-module, where $M$ is a monoid, then $\Sym^d(V)$ is naturally a $KM$-module due to the $KM$-$KS_d$-bimodule structure on $V^{\otimes d}$.  It is well known that if $\rho\colon M\to \mathrm{End}_K(V)$ is the representation afforded by $V$, then \[\theta_{\Sym^d(V)}(m)=h_d(\lambda_1,\ldots, \lambda_n)\] where $h_d(x_1,\ldots, x_n)$ is the complete symmetric polynomial of degree $d$,  $\dim V=n$ and $\lambda_1,\ldots,\lambda_n$ are the eigenvalues of $\rho(m)$ (in a fixed algebraic closure of $K$) with multiplicities, cf.~\cite[Page~77]{Fulton}.    We shall also need the well-known identity~\cite[Appendix~A]{Fulton}
\begin{equation}\label{genfunction}
\sum_{i=0}^{\infty}h_i(x_1,\ldots, x_n)t^i=\prod_{j=1}^n\frac{1}{1-tx_i}.
\end{equation}

\begin{Thm}\label{symmetricpower}
Let $K$ be a field of characteristic $0$, let $M$ be a finite monoid and let $V$ be a finite dimensional $KM$-module affording a faithful representation $\rho\colon M\to \mathrm{End}_K(V)$.  Then every simple $KM$-module is a composition factor of one of $\Sym^0(M),\ldots, \Sym^{r-1}(M)$ with $r=\dim V\cdot s$ where $s$ is the number of distinct characteristic polynomials of the elements $\rho(m)$ with $m\in M$.
\end{Thm}
\begin{proof}
We proceed as in the proof of Theorem~\ref{main}. Let $S$ be a simple $KM$-module and let $e\in M$ be an apex for $S$. Since $\Sym^0(V)$ is the trivial module, we may assume that $S$ is not the trivial module. Then $e$ is not the zero of $M$ (if it has one) and so $eV\neq 0$ because $\rho$ is faithful.  Put $A=KM$ and let $R$ be the radical of $A$. As $eS\neq 0$, there is a primitive idempotent $f$ of $eAe$ such that $f$ is primitive in $A$ and $S\cong Af/Rf$ by Lemma~\ref{corner}.  Write \[f=\sum_{m\in eMe}c_mm.\] The proof of Theorem~\ref{main} shows that $c_e\neq 0$.

Let $a_i = \dim f\Sym^i(V)$ and let $g(t)=\sum_{i=0}^{\infty} a_it^i$ be the corresponding generating function.  We prove that $g(t)$ is a non-zero rational function with denominator of degree at most $r$ by establishing a Molien type formula.

Let $n=\dim V$ and let $p_m(t)$ be the characteristic polynomial of $\rho(m)$ for $m\in M$.  Let $q_1(t),\ldots,q_s(t)$ be the $s$ characteristic polynomials of the endomorphisms $\rho(m)$ with $m\in M$.

Notice that $e\Sym^i(V) = \Sym^i(eV)$ as an $eAe$-module because $eV^{\otimes i} = (eV)^{\otimes i}$.  Let $\rho'\colon eMe\to \mathrm{End}_K(eV)$ be the representation afforded by $eV$.  Note that if $m\in eMe$, then
\begin{equation}\label{eq:switchtoe}
t^np_m(1/t)=\det(I-t\rho(m)) = \det(I-t\rho'(m))
\end{equation}
because if we write $V=eV\oplus (1-e)V$ and choose a basis accordingly, we then have the block form \[I-t\rho(m) =\begin{bmatrix} I-t\rho'(m) & 0\\ 0 & I\end{bmatrix}.\]  Let $M_j =\{m\in eMe\mid p_m(t)=q_j(t)\}$ and assume that $q_1(t)=p_e(t)$.  Let \[b_j = \sum_{m\in M_j}c_m.\] Note that if $M_j=\emptyset$, then $b_j=0$. Observe that \[t^nq_1(1/t) = \det(I-t\rho'(e)) = \det(I-tI)=(1-t)^k\] where $k=\dim V$. On the other hand, since $\rho'$ is faithful if $m\in eMe\setminus \{e\}$, by Lemma~\ref{charkernel} not all eigenvalues of $\rho'(m)$ are $1$. Therefore, $t^np_m(1/t)=\det(I-t\rho'(m))$ is a degree $k$ polynomial whose roots are not all equal to $1$.  In particular, $M_1=\{e\}$ and so $b_1=c_e\neq 0$.

Let $m\in eMe$ and let $\lambda_1,\ldots,\lambda_k$ be the eigenvalues of $\rho'(m)$ with multiplicities in a fixed algebraic closure of $K$.  Then, using \eqref{genfunction}, we have that
\begin{equation*}
\sum_{i=0}^{\infty} \theta_{\Sym^i(eV)}(m)t^i = \sum_{i=0}^{\infty}h_i(\lambda_1,\ldots,\lambda_k)t^i =\prod_{j=1}^k\frac{1}{1-t\lambda_i}=\frac{1}{\det(I-t\rho'(m))}.
\end{equation*}
Therefore, applying \eqref{eq:switchtoe},
\begin{align*}
g(t)&=\sum_{i=0}^{\infty}a_it^i=\sum_{i=0}^{\infty} \theta_{\Sym^i(V)}(f)t^i = \sum_{m\in eMe}c_m\sum_{i=0}^{\infty} \theta_{\Sym^i(eV)}(m)t^i
\\ &= \sum_{m\in eMe} \frac{c_m}{\det(I-t\rho'(m))} = \sum_{j=1}^s \frac{b_j}{t^nq_j(1/t)}= \frac{b_1}{(1-t)^k}+\sum_{j=2}^s\frac{b_j}{t^nq_j(1/t)}
\end{align*}

Since, for all $j=2,\ldots, s$ with $b_j\neq 0$, the polynomial $t^n(q_j(1/t))$ has degree $k$ and not all roots equal to $1$ and since $b_1=c_e\neq 0$, we conclude that $g(t)\neq 0$ and $g(t) = h(t)/q(t)$ where $\deg q(t)\leq ks\leq \dim V\cdot s=r$.  Thus the sequence $a_i$ is not identically zero and satisfies a recurrence of degree $r$, and hence there exists $0\leq i\leq r-1$ such that $a_i\neq 0$.  By Lemma~\ref{multiplicity} we conclude that $S$ is a composition factor of one of $\Sym^0(V),\ldots, \Sym^{r-1}(V)$.
\end{proof}

\begin{Rmk}
Using Newton's identities, the characteristic polynomial of $\rho(m)$ is determined by $\theta_V(m),\ldots, \theta_V(m^{n-1})$ where $n=\dim V$, and hence $s$ can be bounded in terms of the number of values assumed by $\theta_V$.
\end{Rmk}

\begin{Rmk}
Let $V$ and $N_t$ be as in Remark~\ref{nobound}.  Then there are only two distinct characteristic polynomials for elements of $N_t$ acting on $V$ because every non-identity element of $N_t$ acts as a nilpotent operator.  But, for any fixed $k$, $\bigoplus_{i=0}^{k}\Sym^i(V)$ cannot be a faithful  $\mathbb CN_t$-module for  $t$ sufficiently large by dimension considerations.  Thus the smallest $k$ giving a faithful module for the monoid algebra cannot be bounded in terms of just $\dim V$ and the number of different characteristic polynomials, as is the case for finite groups.\end{Rmk}

\section*{Acknowledgments}
Thanks are due to Nicholas Kuhn, who pointed out to me his results~\cite{genericrep,Kuhn2}, which led me to consider symmetric powers.


\def\malce{\mathbin{\hbox{$\bigcirc$\rlap{\kern-7.75pt\raise0,50pt\hbox{${\tt
  m}$}}}}}\def\cprime{$'$} \def\cprime{$'$} \def\cprime{$'$} \def\cprime{$'$}
  \def\cprime{$'$} \def\cprime{$'$} \def\cprime{$'$} \def\cprime{$'$}
  \def\cprime{$'$} \def\cprime{$'$}

\end{document}